\documentclass[reqno]{amsart}
\usepackage{cases}
\usepackage{latexsym}
\usepackage{amsmath}
\usepackage[arrow,matrix]{xy}
\usepackage{stmaryrd}
\usepackage{amsfonts}
\usepackage{amsmath,amssymb,amscd,bbm,amsthm,mathrsfs,dsfont}

\usepackage{fancyhdr}
\usepackage{amsxtra,ifthen}
\usepackage{verbatim}

\theoremstyle{plain}
\newtheorem{theorem}{Theorem}[section]
\newtheorem{lemma}[theorem]{Lemma}

\newtheorem{corollary}[theorem]{Corollary}
\newtheorem{conjecture}[theorem]{Conjecture}

\theoremstyle{definition}
\newtheorem{definition}[theorem]{Definition}
\newtheorem{example}[theorem]{Example}

\newtheorem{remark}[theorem]{Remark}

\begin{document}

\title[Commuting maps on certain incidence algebras]
{Commuting maps on certain incidence algebras}

\author{Hongyu Jia and Zhankui Xiao}

\address{Jia: School of Mathematical Sciences, Huaqiao University,
Quanzhou, Fujian, 362021, P. R. China}
\email{hongyu.jia@hqu.edu.cn}

\address{Xiao: Fujian Province University Key Laboratory of Computation Science,
School of Mathematical Sciences, Huaqiao University,
Quanzhou, Fujian, 362021, P. R. China}

\email{zhkxiao@hqu.edu.cn}

\begin{abstract}
Let $\mathcal{R}$ be a $2$-torsion free commutative ring with unity, $X$ a locally finite pre-ordered set
and $I(X,\mathcal{R})$ the incidence algebra of $X$ over $\mathcal{R}$.
If $X$ consists of a finite number of connected components,
in this paper we give a sufficient and necessary condition for each
commuting map on $I(X,\mathcal{R})$ being proper.
\end{abstract}



\thanks{This work is partially supported by the NSF of Fujian Province (No. 2018J01002)
and the National Natural Science Foundation of China (No. 11301195).}

\maketitle

\section{Introduction}\label{xxsec1}

Let $A$ be an associative algebra over $\mathcal{R}$, a commutative ring with unity. Then
$A$ has the Lie algebra structure under the Lie bracket $[x,y]:=xy-yx$. An $\mathcal{R}$-linear map
$\theta: A\rightarrow A$ is called a {\em commuting map} if $[\theta(x),x]=0$ for all $x\in A$.
A commuting map $\theta$ of $A$ is said to be {\em proper} if it can be written as
$$
\theta(x)=\lambda x+\mu(x), \quad \forall x\in A,
$$
where $\lambda\in Z(A)$, the center of $A$, and $\mu$ is an $\mathcal{R}$-linear map with
image in $Z(A)$. A commuting map which is not proper will be called {\em improper}.
The purpose of this paper is to identify a class of algebras on which every
commuting map is proper.

As far as we know, the first result about commuting maps is due to Posner. He showed in \cite{Posner}
that the existence of a nonzero commuting derivation on a prime algebra $A$ implies the
commutativity of $A$. One of the major promoter for studying commuting maps is Bre\v{s}ar.
He \cite{Bre93-2} initially related the commuting maps to some Herstein's conjectures \cite{Her} which described
the forms of Lie-type maps (for example, Lie isomorphisms, Lie derivations) on associative
simple or prime rings. We warmly encourage the reader to read the well-written survey paper \cite{Bre04},
in which the author presented the development of the theory of commuting maps and their applications,
especially to Lie theory. More results related to commuting maps are considered in
\cite{Bre91,Bre93-1,ChenCai,Cheung,Lee,LeeLee,XiaoWei} etc.

We now recall another notion, incidence algebras \cite{SpDo}, with which we deal in this paper.
Let $(X,\leqslant)$ be a locally finite pre-ordered set (i.e., $\leqslant$ is a reflexive and transitive binary relation).
Here the local finiteness means for any
$x\leqslant y$ there are only finitely many elements $z\in X$ satisfying $x\leqslant z\leqslant y$.
The {\em incidence algebra} $I(X,\mathcal{R})$ of $X$ over $\mathcal{R}$ is defined on the set of functions
$$
I(X,\mathcal{R}):=\{f: X\times X\longrightarrow \mathcal{R}\mid f(x,y)=0\ \text{if}\ x\nleqslant y\}
$$
with the natural $\mathcal{R}$-module structure and the multiplication given by the convolution
$$
(fg)(x,y):=\sum_{x\leqslant z\leqslant y}f(x,z)g(z,y)
$$
for all $f,g\in I(X,\mathcal{R})$ and $x,y\in X$. It would be helpful to claim that the full matrix algebra ${\rm M}_n(\mathcal{R})$,
the upper (or lower) triangular matrix algebras ${\rm T}_n(\mathcal{R})$, and the infinite
triangular matrix algebras ${\rm T}_{\infty}(\mathcal{R})$ are examples of incidence algebras.
The incidence algebra of a partially ordered set was first considered by Ward \cite{Wa}
as a generalized algebra of arithmetic functions. Rota and Stanley developed incidence algebras
as fundamental structures of enumerative combinatorial theory and the allied areas of arithmetic function
theory (see \cite{Stanley}).

Recently, the second author \cite{Xiao} studied the Herstein's Lie-type mapping research program (see \cite{Bre04})
on incidence algebras and proved that every Jordan derivation of $I(X,\mathcal{R})$ degenerates to
a derivation, provided that $\mathcal{R}$ is $2$-torsion free. Since then more and more Lie-type
maps were considered on incidence algebras, see \cite{Khry,WangXiao,ZhangKh} etc. Roughly speaking,
all the known Lie-type maps of $I(X,\mathcal{R})$ are proper or of the standard form. On the other hand,
it is well-known that the commuting maps of the triangular matrix algebras ${\rm T}_n(\mathcal{R})$
and the full matrix algebra ${\rm M}_n(\mathcal{R})$ are proper (see \cite{Cheung,XiaoWei}). Therefore,
it is hopefully that every commuting map of the incidence algebra $I(X,\mathcal{R})$ is proper.
Unfortunately, we can find a counter-example as follows. Let $X$ be the partially ordered set with
Hasse diagram $\overset{1}{\circ}\leftarrow\overset{3}{\circ}\rightarrow\overset{2}{\circ}$. Then
$$
I(X,\mathcal{R})=\left\{\left. \left(
\begin{array}
[c]{ccc}%
a_{11} & 0 & a_{13}\\
0 & a_{22} & a_{23}\\
0 & 0 & a_{33}\\
\end{array}
\right) \right|\ a_{ij}\in\mathcal{R} \right\}.
$$
The reader can verify that the $\mathcal{R}$-linear map $\theta$ defined by
$$
\theta: \left(
\begin{array}
[c]{ccc}%
a_{11} & 0 & a_{13}\\
0 & a_{22} & a_{23}\\
0 & 0 & a_{33}\\
\end{array}
\right)\longmapsto \left(
\begin{array}
[c]{ccc}%
a_{11} & 0 & a_{13}\\
0 & a_{33} & 0\\
0 & 0 & a_{33}\\
\end{array}
\right)
$$
is a commuting map, but generally improper. Hence it is fascinating to study the relationship
between commuting maps and the structure of incidence algebras, especially the combinatorial
structure. This is our main motivation of this paper.

\section{The finite rank case}\label{xxsec2}

\emph{From now on, we always assume that all the rings and algebras are $2$-torsion free}.
In this section, we study commuting maps of the incidence algebra $I(X,\mathcal{R})$ when $X$ is a {\em finite} pre-ordered set.
For reader's convenience, let's start with two well-known results for general algebras.

\begin{lemma}\label{sec2.1} {\rm (\cite[Proposition 2]{Cheung})}
Let $A,B$ be two $\mathcal{R}$-algebras.
Then $A$ and $B$ have no improper commuting maps if and only if $A\oplus B$ has no improper commuting maps.
\end{lemma}

\begin{lemma}\label{sec2.2}
Let $A$ be an $\mathcal{R}$-algebra with an $\mathcal{R}$-linear basis $Y$. Then an $\mathcal{R}$-linear map
$\theta: A\rightarrow A$ is a commuting map if and only if
$[\theta(x),y]=[x,\theta(y)]$ for all $x,y\in Y$.
\end{lemma}

We now introduce some standard notations for the incidence algebra $I(X,\mathcal{R})$.
The unity element $\delta$ of $I(X,\mathcal{R})$ is given by $\delta(x,y)=\delta_{xy}$ for $x\leqslant y$,
where $\delta_{xy}\in \{0,1\}$ is the Kronecker delta. If $x,y\in X$ with $x\leqslant y$, let $e_{xy}$ be
defined by $e_{xy}(u,v)=1$ if $(u,v)=(x,y)$, and $e_{xy}(u,v)=0$ otherwise. Then $e_{xy}e_{uv}=\delta_{yu}e_{xv}$
by the definition of convolution. Clearly the set $\mathfrak{B}:=\{e_{xy}\mid x\leqslant y\}$ forms an $\mathcal{R}$-linear
basis of $I(X,\mathcal{R})$ when $X$ is finite. For $i\leqslant j$ and $i\neq j$, we write $i<j$ or $j>i$ for short.

When $X$ is finite, $I(X,\mathcal{R})$ is also known as a digraph algebra. This means that
there is a directed graph with the vertex set $X$ associated with $I(X,\mathcal{R})$. This graph contains all the
self loops and the matrix unit $e_{xy}$ corresponds to a directed edge from $y$ to $x$.
We now define an equivalent relation on $\mathfrak{D}:=\{(x,y)\mid x< y\}$, the set of all directed edges
associated with the pre-ordered set $X$.
Let $\sim$ be the relation on $X$ defined by $x\sim y$ if and only if $x\leqslant y$ or $y\leqslant x$, in other words,
there is at least one directed edge between the vertices $x$ and $y$.
We also write $y\sim x$ for $x\sim y$ by the symmetry.
Let $x_1,x_2,\ldots,x_n$ be $n$ different vertices in $X$ with $n\geq 2$. We say that
$\{x_1,x_2,\ldots,x_n\}$ forms a {\em circle} if
\begin{enumerate}
\item[(i)] $x_1\sim x_2$ when $n=2$;

\item[(ii)] $x_1\sim x_2$, \ldots, $x_{i-1}\sim x_i$, \ldots, $x_{n-1}\sim x_n$, $x_n\sim x_1$ when $n\geq 3$.
\end{enumerate}

\begin{definition}\label{sec2.3}
For any two directed edges $(x,y)$ and $(u,v)$ in $\mathfrak{D}$, we define $(x,y)\cong (u,v)$ if
and only if there is a circle contains both $x\sim y$ and $u\sim v$.
\end{definition}

The reader can verify that the binary relation $\cong$ is in deed an equivalent relation
on the set $\mathfrak{D}$. We remark here that the $n=2$ case does not satisfy the standard
notion, circle, in combinatorial theory, but it provides the reflexivity of the relation
$\cong$. The example appeared in the introduction, where $(1,3)\ncong (2,3)$, and the following Example \ref{sec2.4}
can help us get a more clear picture for the relation $\cong$.

\begin{example}\label{sec2.4}
Let $X=\{1,2,3,4\}$ with partially ordered relations (or arrows) $1<2, 2<3, 2<4$. Then the associated set of
directed edges is $\mathfrak{D}=\{(1,2), (2,3), (1,3),\\ (2,4), (1,4)\}$. We have $(2,3)\cong (2,4)$,
since $2\sim 3, 3\sim 1, 1\sim 4, 4\sim 2$, and $\mathfrak{D}$ forms one equivalent class
under the relation $\cong$.
\end{example}

The main result of this section is as follows.

\begin{theorem}\label{main in section 2}
Let $\mathcal{R}$ be a $2$-torsion free commutative ring with unity and $X$ be finite.
Then every commuting map $\theta$ of $I(X,\mathcal{R})$ is proper if and only if
the set of all directed edges associated with each connected component forms
one equivalent class under the relation $\cong$.
\end{theorem}

We only need to prove Theorem \ref{main in section 2} when $X$ is connected. In fact, assume that
$X=\bigsqcup_{i\in I}X_i$ be the union of its distinct connected components, where $I$ is a finite index set.
Let $\delta_i:=\sum_{x\in X_i} e_{xx}$. It follows from \cite[Theorem 1.3.13]{SpDo} that $\{\delta_i\mid i\in I\}$ forms
a complete set of central primitive idempotents. In other words, $I(X,\mathcal{R})=\bigoplus_{i\in I}\delta_i
I(X,\mathcal{R})$. Clearly $\delta_i I(X,\mathcal{R})\cong I(X_i,\mathcal{R})$ for each $i\in I$.
Hence we only need to prove Theorem \ref{main in section 2} when $X$ is connected by Lemma \ref{sec2.1}.

From now on until the end of this section, we assume $X$ is finite and connected.
Let $\theta: I(X,\mathcal{R})\rightarrow I(X,\mathcal{R})$ be an arbitrary commuting map. We denote
for all $i,j\in X$ with $i\leqslant j$
$$
\theta(e_{ij})=\sum_{e_{xy}\in \mathfrak{B}}C_{xy}^{ij}e_{xy}.
$$
We also make the convention $C_{xy}^{ij}=0$, if needed, for $x\nleqslant y$.

\begin{lemma}\label{sec2.6}
The commuting map $\theta$ satisfies
\begin{align}
\theta(e_{ii})&=\sum_{x\in X}C_{xx}^{ii}e_{xx};\label{(1)}\\
\theta(e_{ij})&=\sum_{x\in X}C_{xx}^{ij}e_{xx}+C_{ij}^{ij}e_{ij}, \hspace{6pt} \text{if}\ i\neq j.\label{(2)}
\end{align}
\end{lemma}

\begin{proof}
Without loss of generality, we assume that $|X|\geq 2$. Since $[\theta(e_{ii}),e_{xx}]=[e_{ii},\theta(e_{xx})]$
for any $i,x\in X$, multiplying this identity by $e_{xx}$ on the left and by $e_{yy}$ on the right, we have
$$
C^{ii}_{xy}=0,\ \text{if}\ i\neq x<y\neq i.\eqno(3)
$$
Then left multiplication by $e_{xx}$ on $[\theta(e_{ii}),e_{ii}]=0$ leads to
$$
C^{ii}_{xi}=0,\ \text{if}\ x<i.\eqno(4)
$$
Similarly, right multiplication by $e_{yy}$ on $[\theta(e_{ii}),e_{ii}]=0$ leads to
$$
C^{ii}_{iy}=0,\ \text{if}\ i<y.\eqno(5)
$$
Combining the identities (3), (4) and (5), we obtain
$$\begin{aligned}
\theta(e_{ii})&=\sum_{x\leqslant y}C_{xy}^{ii}e_{xy}=\sum_{x\in X}C_{xx}^{ii}e_{xx}+\sum_{x<y}C_{xy}^{ii}e_{xy}\\
&=\sum_{x\in X}C_{xx}^{ii}e_{xx}+\sum_{i\neq x<y\neq i}C_{xy}^{ii}e_{xy}+\sum_{x<i}C_{xi}^{ii}e_{xi}+\sum_{i<y}C_{iy}^{ii}e_{iy}\\
&=\sum_{x\in X}C_{xx}^{ii}e_{xx},
\end{aligned}$$
which proves (1).

Noting that $[\theta(e_{ij}),e_{yy}]=[e_{ij},\theta(e_{yy})]$
for any $i<j$ and $y\in X$, multiplying this identity by $e_{xx}$ on the left and by $e_{yy}$ on the right, we have
$$
C^{ij}_{xy}=0,\ \text{if}\ i\neq x<y\neq j.\eqno(6)
$$
Taking the identity (1) into account, then left multiplication by $e_{xx}$ on $[\theta(e_{ij}),e_{jj}]=[e_{ij},\theta(e_{jj})]$ leads to
$$
C^{ij}_{xj}=0,\ \text{if}\ i\neq x<j.\eqno(7)
$$
Similarly, right multiplication by $e_{yy}$ on $[\theta(e_{ij}),e_{ii}]=[e_{ij},\theta(e_{ii})]$ leads to
$$
C^{ij}_{iy}=0,\ \text{if}\ i<y\neq j.\eqno(8)
$$
Combining the identities (6), (7) and (8), we obtain
$$\begin{aligned}
\theta(e_{ij})&=\sum_{x\leqslant y}C_{xy}^{ij}e_{xy}=\sum_{x\in X}C_{xx}^{ij}e_{xx}+\sum_{x<y}C_{xy}^{ij}e_{xy}\\
&=\sum_{x\in X}C_{xx}^{ij}e_{xx}+\sum_{i\neq x<y\neq j}C_{xy}^{ij}e_{xy}+\sum_{i\neq x<j}C_{xj}^{ij}e_{xj}+
   \sum_{i<y\neq j}C_{iy}^{ij}e_{iy}+C_{ij}^{ij}e_{ij}\\
&=\sum_{x\in X}C_{xx}^{ij}e_{xx}+C_{ij}^{ij}e_{ij},
\end{aligned}$$
which proves (2).
\end{proof}

\begin{lemma}\label{sec2.7}
Let $X$ be connected and $\theta: I(X,\mathcal{R})\rightarrow I(X,\mathcal{R})$ be an $\mathcal{R}$-linear
map satisfying the formulas ${\rm (\ref{(1)})}$ and ${\rm (\ref{(2)})}$. Then $\theta$ is a commuting map if and only if
the coefficients $C^{ij}_{xy}$ are subject to the following relations:
$$\begin{aligned}
&{\rm (R1)}\hspace{6pt} C_{il}^{il}=C_{ii}^{ii}-C_{ll}^{ii},  &&\text{if}\hspace{6pt} i<l;\\
&{\rm (R2)}\hspace{6pt} C_{ki}^{ki}=C_{ii}^{ii}-C_{kk}^{ii},  &&\text{if}\hspace{6pt} k<i;\\
&{\rm (R3)}\hspace{6pt} C_{kk}^{ii}=C_{ll}^{ii},  &&\text{if}\hspace{6pt} k<l\ \text{and}\ k\neq i\neq l;\\
&{\rm (R4)}\hspace{6pt} C_{xx}^{ij}=C_{yy}^{ij}, &&\forall\hspace{6pt} x,y\in X;\\
&{\rm (R5)}\hspace{6pt} C_{ij}^{ij}=C_{jl}^{jl},  &&\text{if}\hspace{6pt} i<j<l.
\end{aligned}$$
\end{lemma}

\begin{proof}
By Lemma \ref{sec2.2}, we have $\theta$ is a commuting map if and only if
$[\theta(e_{ij}),e_{kl}]=[e_{ij},\theta(e_{kl})]$ for all $i\leqslant j$ and $k\leqslant l$.
There are four cases occurring depending on the formulas ${\rm (\ref{(1)})}$ or ${\rm (\ref{(2)})}$.

{\bf Case 1}. If $i=j$ and $k=l$, then $[\theta(e_{ii}),e_{kk}]=[e_{ii},\theta(e_{kk})]$
always holds.

{\bf Case 2}. If $i=j$ and $k\neq l$, then the formulas ${\rm (\ref{(1)})}$ and ${\rm (\ref{(2)})}$ imply that
$$
(C_{kk}^{ii}-C_{ll}^{ii})e_{kl}=(\delta_{ik}-\delta_{il})C_{kl}^{kl}e_{kl}.\eqno(9)
$$
Noting that $k\neq l$, hence we can get (R1) (resp. (R2)) from the identity (9) by setting $i=k$ (resp. $i=l$).
We can also get the relation (R3) from (9) by setting $k\neq i\neq l$.

{\bf Case 3}. If $i\neq j$ and $k=l$, this case is symmetric to the {\bf Case 2}.

{\bf Case 4}. If $i\neq j$ and $k\neq l$, it follows from the formula ${\rm (\ref{(2)})}$ that
$$
(C_{kk}^{ij}-C_{ll}^{ij})e_{kl}+\delta_{jk}(C_{ij}^{ij}-C_{kl}^{kl})e_{il}=
(C_{jj}^{kl}-C_{ii}^{kl})e_{ij}+\delta_{il}(C_{ij}^{ij}-C_{kl}^{kl})e_{kj}.\eqno(10)
$$
When $j\neq k$ and $i\neq l$, the identity (10) is equivalent to
$$
C_{kk}^{ij}=C_{ll}^{ij} \hspace{6pt} \text{for}\ j\neq k< l\neq i.\eqno(11)
$$
When $j\neq k$ and $i=l$ (hence $k<i<j$), the identity (10) can be rewritten as
$(C_{kk}^{ij}-C_{ii}^{ij})e_{ki}=(C_{jj}^{ki}-C_{ii}^{ki})e_{ij}+(C_{ij}^{ij}-C_{ki}^{ki})e_{kj}$,
which gives
$$
C_{kk}^{ij}=C_{ii}^{ij} \hspace{6pt} \text{for}\ j\neq k<i;\eqno(12)
$$
$$
C_{ii}^{ki}=C_{jj}^{ki} \hspace{6pt} \text{for}\ i<j\neq k;\eqno(13)
$$
$$
C_{ki}^{ki}=C_{ij}^{ij} \hspace{6pt} \text{for}\ j\neq k<i<j.\eqno(14)
$$
For convenience, substituting the indices we here rewrite the identity (13) as
$$
C_{jj}^{ij}=C_{ll}^{ij} \hspace{6pt} \text{for}\ j<l\neq i.\eqno(15)
$$
When $j=k$ and $i\neq l$, by symmetry we can also get the identities (12)-(15).
When $j=k$ and $i=l$ (hence $i<j<i$), the identity (10) is equivalent to
$(C_{jj}^{ij}-C_{ii}^{ij})e_{ji}+(C_{ij}^{ij}-C_{ji}^{ji})e_{ii}=
(C_{jj}^{ji}-C_{ii}^{ji})e_{ij}+(C_{ij}^{ij}-C_{ji}^{ji})e_{jj}$, which gives
$$
C_{jj}^{ij}=C_{ii}^{ij} \hspace{6pt} \text{for}\ j<i;\eqno(16)
$$
$$
C_{ij}^{ij}=C_{ji}^{ji} \hspace{6pt} \text{for}\ i<j<i.\eqno(17)
$$
Combining (14) and (17), we obtain the relation (R5). The identities (11), (12), (15) and (16) imply that
$$
C_{xx}^{ij}=C_{yy}^{ij} \hspace{6pt} \text{for all}\ x<y.\eqno(18)
$$
Since $X$ is connected, for any $x,y\in X$ there is a sequence $x=x_0,x_1,\ldots,x_s=y$ in $X$ such that
$x_{i-1}$ covers or is covered by $x_{i}$ for $1\leq i\leq s$. Then a recursive procedure,
using (18), on the length $s$ implies the desired relation (R4).
\end{proof}

\begin{remark}\label{sec2.8}
The reader may find that the set of the relations (R1-R5) in Lemma \ref{sec2.7} is not minimal. For example,
the relations (R1), (R2) and (R5) can imply the relation (R3). However, we present here the Lemma \ref{sec2.7} for
later use.
\end{remark}

\begin{proof}[Proof of Theorem \ref{main in section 2}]
Let $\theta$ be an arbitrary commuting map of the incidence algebra
$I(X,\mathcal{R})$. Then $\theta$ has the form ${\rm (\ref{(1)})}$ and ${\rm (\ref{(2)})}$ in Lemma \ref{sec2.6}, where
the coefficients $C^{ij}_{xy}$ satisfy the relations (R1-R5) in Lemma \ref{sec2.7}.
If $|X|=1$, the incidence algebra $I(X,\mathcal{R})\cong \mathcal{R}$, and hence Theorem \ref{main in section 2} is clear.
We now assume $|X|\geq 2$ and first study the sufficiency.

We claim that the relation (R3) can be strengthened as $C_{xx}^{ii}=C_{yy}^{ii}$ for $x\neq i\neq y$.
Since the set of all directed edges associated with $X$ (is assumed connected) forms
one equivalent class under the relation $\cong$, there is a circle contains both
$x$ and $y$. Breaking the circle at $x$ and $y$,
there exists a sequence $x=x_1,x_2,\ldots,x_s=y$ of different vertices in $X$ such that
$x_1\sim x_2, \ldots, x_i\sim x_{i+1}, \ldots, x_{s-1}\sim x_s$ and $i\notin \{x_1,x_2,\ldots,x_s\}$.
A recursive procedure, using (R3), on the length $s$ implies the desired claim.

We then claim that $C_{ij}^{ij}=C_{kl}^{kl}$ for all $i<j$ and $k<l$.
By (R1), we have $C_{ij}^{ij}+C_{jj}^{ii}=C_{ii}^{ii}=C_{il}^{il}+C_{ll}^{ii}$ and
we just proved that $C_{jj}^{ii}=C_{ll}^{ii}$ for $j\neq i\neq l$. Hence
$$
C_{ij}^{ij}=C_{il}^{il} \hspace{6pt} \text{for}\ i<j\ \text{and}\ i<l.\eqno(19)
$$
Similarly, we have
$$
C_{ij}^{ij}=C_{kj}^{kj} \hspace{6pt} \text{for}\ i<j\ \text{and}\ k<j.\eqno(20)
$$
Since $X$ is connected, there is a path (forget the direction) form $j$ to $k$, i.e.,
there exists a series $j=x_1,x_2,\ldots,x_s=k$ in $X$ such that
$x_1\sim x_2, \ldots, x_i\sim x_{i+1}, \ldots, x_{s-1}\sim x_s$.
Using (19), (20) and the relation (R5), an induction on the index $s$ implies the desired claim.

For any $i\in X$, since $X$ is connected and $|X|\geq 2$, there is a vertex comparable with $i$.
Let's define an $\mathcal{R}$-linear map $L: I(X,\mathcal{R})\rightarrow I(X,\mathcal{R})$ by
$L(e_{ij})=C_{ij}^{ij}e_{ij}$ for $i<j$ and $L(e_{ii})=C_{ij}^{ij}e_{ii}$ if there exists a
vertex $j$ such that $i<j$, or $L(e_{ii})=C_{ki}^{ki}e_{ii}$ if there exists a
vertex $k$ such that $k<i$. Note that $C_{ij}^{ij}=C_{kl}^{kl}$ for all $i<j$ and $k<l$.
The map $L$ is well-defined and is of the form $L(f)=\lambda f$, where $\lambda=C_{ij}^{ij}$
for $i<j$ and $f\in I(X,\mathcal{R})$.

Let $\mu:=\theta-L$. If there is a vertex $j$ such that $i<j$, then
$\mu(e_{ii})=\sum_{x\neq i}C_{xx}^{ii}e_{xx}+(C_{ii}^{ii}-C_{ij}^{ij})e_{ii}$.
Combining the fact $C_{xx}^{ii}=C_{yy}^{ii}$ for $x\neq i\neq y$ and the relation (R1),
we get $\mu(e_{ii})\in Z(I(X,\mathcal{R}))$, the center of $I(X,\mathcal{R})$. Similarly,
we can get $\mu(e_{ii})\in Z(I(X,\mathcal{R}))$ if there is a vertex $k$ such that $k<i$.
Finally the relation (R4) implies $\mu(e_{ij})\in Z(I(X,\mathcal{R}))$ for all $i<j$.
Hence $\mu$ is a central-valued linear map and $\theta$ is proper.

We now prove the necessity. By the hypothesis, $\theta$ is proper. Hence
$C_{ij}^{ij}=C_{kl}^{kl}$ for all $i<j$ and $k<l$.
For any $i\in X$, since $X$ is connected and $|X|\geq 2$, there is a vertex comparable with $i$.
Therefore, either
$C_{ii}^{ii}-C_{ij}^{ij}=C_{xx}^{ii}$ for $x\neq i<j$ or $C_{ii}^{ii}-C_{ki}^{ki}=C_{xx}^{ii}$ for $k<i\neq x$.
We can get $C_{xx}^{ii}=C_{yy}^{ii}$ for $x\neq i\neq y$ from each case.
If the set of all directed edges associated with $X$ does not form
one equivalent class under the relation $\cong$, then there exist two vertices $j$ and $k$
such that they can not be contained in any circle. In this case, we must have $|X|\geq 3$. In fact,
when $|X|=2$, every commuting map of $I(X,\mathcal{R})$ is proper and the set of all directed edges associated with $X$ forms
one equivalent class. Fix one vertex $i$ which is different from $j$ and $k$.
The connectivity of $X$ shows that there is a path (forget the direction) form
$i$ to $j$ (resp. a path from $i$ to $k$). By (R3), we can assume $j$ and $k$
are comparable to $i$ without lose of generality.
Noting that $j$ and $k$ do not be contained in any circle,
we can construct, by Lemmas \ref{sec2.6} and \ref{sec2.7}, a commuting map $\theta$ such that $C_{jj}^{ii}\neq C_{kk}^{ii}$,
a contradiction. We completes the proof of Theorem \ref{main in section 2}.
\end{proof}

From Theorem \ref{main in section 2}, we immediately have the following corollaries.

\begin{corollary}\label{sec2.9}
Let $\mathcal{R}$ be a $2$-torsion free commutative ring with unity and
${\rm T}_n(\mathcal{R})$ be the upper (or lower) triangular matrix algebra over $\mathcal{R}$.
Then every commuting map of ${\rm T}_n(\mathcal{R})$ is proper.
\end{corollary}

\begin{corollary}\label{sec2.10}
Let $\mathcal{R}$ be a $2$-torsion free commutative ring with unity and
${\rm M}_n(\mathcal{R})$ be the full matrix algebra over $\mathcal{R}$.
Then every commuting map of ${\rm M}_n(\mathcal{R})$ is proper.
\end{corollary}

\section{The general case}\label{xxsec3}

In this section, we extend the Theorem \ref{main in section 2} to the case of $X$ being a locally finite pre-ordered set.
Let $\tilde{I}(X,\mathcal{R})$ be the $\mathcal{R}$-subspace of $I(X,\mathcal{R})$
generated by the elements $e_{xy}$ with $x\leqslant y$. That means $\tilde{I}(X,\mathcal{R})$
consists exactly of the functions $f\in I(X,\mathcal{R})$ which are nonzero only at a finite number of $(x,y)$.
Clearly $\tilde{I}(X,\mathcal{R})$ is a subalgebra of $I(X,\mathcal{R})$.
Hence $I(X,\mathcal{R})$ becomes an $\tilde{I}(X,\mathcal{R})$-bimodule in the natural manner.
Let $\theta: \tilde{I}(X,\mathcal{R})\rightarrow I(X,\mathcal{R})$ be a commuting map, i.e.
$$
[\theta(f),f]=0
$$
for all $f\in \tilde{I}(X,\mathcal{R})$.
Observe that Lemmas \ref{sec2.6} and \ref{sec2.7} remain valid, when we replace the domain
of $\theta$ by $\tilde{I}(X,\mathcal{R})$. In fact, although the formal sums $L(e_{ij})=\sum_{x\leqslant y}C_{xy}^{ij}e_{xy}$ are now infinite,
multiplication by $e_{uv}$ on the left or on the right works as in the finite case.

Let's now recall some notations and results from \cite{ZhangKh}.
For any $f\in I(X,\mathcal{R})$ and $x\leqslant y$, {\it the restriction of $f$} to $\{z\in X\mid x\leqslant z\leqslant y\}$ is defined by
$$
f|_{x}^{y}=\sum _{x\leqslant u\leqslant v\leqslant y}f(u,v)e_{uv}
$$
Observe that the sum above is finite, and hence $f|_{x}^{y}\in \tilde{I}(X,\mathcal{R})$.
For any $f\in I(X,\mathcal{R})$ and $x\leqslant y$, the observation
$$
e_{xx}fe_{yy}=f(x,y)e_{xy} \eqno(21)
$$
will be frequently used.

\begin{lemma}\label{sec3.1} {\rm (\cite[Lemma 3.2 (ii)]{ZhangKh})}
For any $f\in I(X,\mathcal{R})$ and $x\leqslant y$, we have
$$
(fg)(x,y)=(f|_{x}^{y} g)(x,y)=(fg|_{x}^{y})(x,y)=(f|_{x}^{y} g|_{x}^{y})(x,y).
$$
\end{lemma}

\begin{lemma}\label{sec3.2}
Let $\theta$ be a commuting map of $I(X,\mathcal{R})$ and $x<y$. Then
$$
\theta(f)(x,y)=\theta(f|_{x}^{y})(x,y).
$$
\end{lemma}

\begin{proof}
It follows from $(21)$ that
$$\begin{aligned}
\theta(f)(x,y)&=[e_{xx},\theta(f)](x,y)=[\theta(e_{xx}),f](x,y)\\
&=[\theta(e_{xx}),f|_{x}^{y}](x,y)=\theta(f|_{x}^{y})(x,y),
\end{aligned}$$
where the third identity relies on Lemma \ref{sec3.1}.
\end{proof}

For the locally finite pre-ordered set $X$, we set $\mathfrak{D}:=\{(x,y)\mid x< y\}$
and also define the equivalent relation $\cong$ on $\mathfrak{D}$ as those in
Definition \ref{sec2.3}.

\begin{theorem}\label{sec3.3}
Let $\theta$ be a commuting map of $I(X,\mathcal{R})$. If $\mathfrak{D}$ forms
one equivalent class under the relation $\cong$, then $\theta$ is proper.
\end{theorem}

\begin{proof}
Restricting $\theta$ to $\tilde{I}(X,\mathcal{R})$, denoting as $\theta_{\tilde{I}}$,
we get that $\theta_{\tilde{I}}: \tilde{I}(X,\mathcal{R})\rightarrow I(X,\mathcal{R})$
is a commuting map. Then $\theta_{\tilde{I}}$ is proper by Theorem \ref{main in section 2}.
By the hypothesis, $X$ is connected and hence $Z(I(X,\mathcal{R}))\cong \mathcal{R}$ by
\cite[Corollary 1.3.15]{SpDo}. There exists $\lambda\in \mathcal{R}$ such that
$\theta_{\tilde{I}}(f)=\lambda f+\tilde{\mu}(f)$, where $f\in \tilde{I}(X,\mathcal{R})$ and
$\tilde{\mu}: \tilde{I}(X,\mathcal{R})\rightarrow Z(I(X,\mathcal{R}))$ is an $\mathcal{R}$-linear map.

Define $\mu: I(X,\mathcal{R})\rightarrow I(X,\mathcal{R})$ by
$\mu(f):=\theta(f)-\lambda f$
for all $f\in I(X,\mathcal{R})$. We only need to prove the commuting map $\mu$ is central-valued.
Notice that $\mu(f)=\tilde{\mu}(f)\in Z(I(X,\mathcal{R}))$ for all $f\in \tilde{I}(X,\mathcal{R})$.
For $x<y$ and $f\in I(X,\mathcal{R})$, Lemma \ref{sec3.2} shows $\mu(f)(x,y)=\mu(f|_{x}^{y})(x,y)
=\tilde{\mu}(f|_{x}^{y})(x,y)=0$. We now prove $\mu(f)(x,x)=\mu(f)(y,y)$
for all $f\in I(X,\mathcal{R})$ and $x,y\in X$.
Since $X$ is connected, we only need to prove $\mu(f)(x,x)=\mu(f)(y,y)$
for $x<y$. On one hand, by (21),
$$
[e_{xy},\mu(f)](x,y)=\mu(f)(y,y)-\mu(f)(x,x).
$$
On the other hand, by Lemma \ref{sec3.1},
$$\begin{aligned}
&\hspace{12pt}[e_{xy},\mu(f)](x,y)=[\mu(e_{xy}),f](x,y)=[\mu(e_{xy}),f|_{x}^{y}](x,y)\\
&=[e_{xy},\mu(f|_{x}^{y})](x,y)=\mu(f|_{x}^{y})(y,y)-\mu(f|_{x}^{y})(x,x)\\
&=\tilde{\mu}(f|_{x}^{y})(y,y)-\tilde{\mu}(f|_{x}^{y})(x,x)=0.
\end{aligned}$$
Therefore, it follows from \cite[Theorem 1.3.13]{SpDo} that
$\mu$ is central-valued and hence $\theta$ is proper.
\end{proof}

\begin{corollary}\label{sec3.4}
Let $\mathcal{R}$ be a $2$-torsion free commutative ring with unity. Let
${\rm T}_{\infty}(\mathcal{R})$ be the ring of countable upper triangular $\mathcal{R}$-matrices.
Then every commuting map of ${\rm T}_{\infty}(\mathcal{R})$ is proper.
\end{corollary}

It is obvious that Theorem \ref{sec3.3} can be generalized to the case of
$X$ consisting of a finite number of connected components. Hence the following conjecture is natural.

\begin{conjecture}
Let $(X,\leqslant)$ be a locally finite pre-ordered set and $\mathcal{R}$ be $2$-torsion free.
If the set $\mathfrak{D}$ associated with each connected component forms
one equivalent class under the relation $\cong$,
then every commuting map of $I(X,\mathcal{R})$ is proper.
\end{conjecture}


\end{document}